\documentclass{amsart}
\usepackage{amsmath,amssymb,amsthm,url,comment}
\usepackage{graphicx,color}

\newtheorem{theorem}{Theorem}[section]
\newtheorem{lemma}[theorem]{Lemma} 
 
\newtheorem{proposition}[theorem]{Proposition} 
\newtheorem{conjecture}{Conjecture} 
 
\newtheorem{corollary}[theorem]{Corollary}
\theoremstyle{definition}

\newtheorem{example}[theorem]{Example}
\newtheorem{remark}[theorem]{Remark}
\newtheorem*{remark*}{Remark}

\numberwithin{equation}{section}

\newcommand{\Z}{\mathbb{Z}}
\newcommand{\R}{\mathbb{R}}
\newcommand{\C}{\mathbb{C}}
\newcommand{\Q}{\mathbb{Q}}

\newcommand{\Thetatwodia}{
\raisebox{-6mm}{
\begin{picture}(20,38)
\put(10,20){\oval(20,28)}
\put(0,25){\line(1,0){20}}
\put(0,15){\line(1,0){20}}
\end{picture}
}}

\newcommand{\Thetathreedia}{
\raisebox{-6mm}{
\begin{picture}(20,38)
\put(10,20){\oval(20,28)}
\put(1,28){\line(1,0){18}}
\put(0,20){\line(1,0){20}}
\put(1,12){\line(1,0){18}}
\end{picture}
}}

\newcommand{\Atwodia}{
\raisebox{-4mm}{
\begin{picture}(20,32)
\put(10,16){\oval(26,16)}
\put(2,24){\line(0,1){10}}
\put(16,24){\line(0,1){10}}
\end{picture}
}}

\newcommand{\Afourdia}{
\raisebox{-4mm}{
\begin{picture}(20,32)
\put(10,16){\oval(26,16)}
\put(2,24){\line(0,1){10}}
\put(12,24){\line(0,1){10}}
\put(7,24){\line(0,1){10}}
\put(17,24){\line(0,1){10}}
\end{picture}
}}

\newcommand{\Asixdia}{
\raisebox{-4mm}{
\begin{picture}(28,32)
\put(14,16){\oval(32,16)}
\put(4,24){\line(0,1){10}}
\put(8,24){\line(0,1){10}}
\put(12,24){\line(0,1){10}}
\put(16,24){\line(0,1){10}}
\put(20,24){\line(0,1){10}}
\put(24,24){\line(0,1){10}}
\end{picture}
}}

\newcommand{\Btwodia}{
\raisebox{-4mm}{
\begin{picture}(20,32)
\put(10,16){\oval(26,16)}
\put(2,24){\line(0,1){10}}
\put(16,24){\line(0,1){10}}
\put(-3,16){\line(1,0){26}}
\end{picture}
}}

\newcommand{\Bfourdia}{
\raisebox{-4mm}{
\begin{picture}(20,32)
\put(10,16){\oval(26,16)}
\put(2,24){\line(0,1){10}}
\put(12,24){\line(0,1){10}}
\put(7,24){\line(0,1){10}}
\put(17,24){\line(0,1){10}}
\put(-3,16){\line(1,0){26}}
\end{picture}
}}

\newcommand{\Ctwodia}{
\raisebox{-4mm}{
\begin{picture}(20,32)
\put(10,16){\oval(26,16)}
\put(2,24){\line(0,1){10}}
\put(16,24){\line(0,1){10}}
\put(-2,19){\line(1,0){24}}
\put(-2,13){\line(1,0){24}}
\end{picture}
}}

\begin{document}

\title[Constraints for chirally cosmetic surgeries]{On constraints for knots to admit chirally cosmetic surgeries and their calculations}

\author{Kazuhiro Ichihara}
\address{Department of Mathematics, College of Humanities and Sciences, Nihon University, 3-25-40 Sakurajosui, Setagaya-ku, Tokyo 156-8550, JAPAN}
\email{ichihara.kazuhiro@nihon-u.ac.jp}
\author{Tetsuya Ito}
\address{Department of Mathematics, Kyoto University, Kyoto 606-8502, JAPAN}
\email{tetitoh@math.kyoto-u.ac.jp}

\author{Toshio Saito}
\address{Department of Mathematics, Joetsu University of Education, 1 Yamayashiki, Joetsu 943-8512, JAPAN}
\email{toshio@juen.ac.jp}

\subjclass[2020]{Primary~57K10,57K18,57K31}

\begin{abstract}
We discuss various constraints for knots in $S^{3}$ to admit chirally cosmetic surgeries, derived from invariants of 3-manifolds, such as, the quantum $SO(3)$-invariant, the rank of the Heegaard Floer homology, and finite type invariants. We apply them to show that a large portion (roughly 75\%) of knots which are neither amphicheiral nor $(2,p)$-torus knots with less than or equal to 10 crossings admits no chirally cosmetic surgeries.
\end{abstract}

\maketitle
\section{Introduction}

For a knot $K$ in a 3-manifold $M$ and a slope $r$, we denote by $M_K(r)$ the $r$-surgery on a knot $K$. When $M=S^{3}$ we identify the set of slopes by $\Q \cup \{\infty = \frac{1}{0}\}$ as usual. In the following, we always treat non-meridional and non-longitudinal slopes so we regard a slope as a non-zero rational number $r=\frac{m}{n}$. Throughout the paper, we use the convention that a slope $\frac{m}{n}$ is always expressed by coprime integers $m,n$ such that $m>0$.

Two slopes are called \emph{inequivalent} if there are no orientation-preserving homeomorphisms of $E(K)$ that send one to the other. When $K$ is not the unknot, then two slopes $r,r' \in \Q$ are inequivalent if and only if $r \neq r'$ \cite{gl}. Two Dehn surgeries $S^{3}_K(r)$ and $S^{3}_K(r')$ on inequivalent slopes $r,r'$ are called \emph{purely cosmetic} (resp. \emph{chirally cosmetic}) if $S^{3}_K(r) \cong S^{3}_{K}(r')$ (resp. $S^{3}_K(r) \cong -S^{3}_{K}(r')$).
Here $-M$ represents the closed oriented 3-manifold $M$ with opposite orientation, and we denote by $M \cong M'$ if there is an orientation-preserving homeomorphism $f: M\rightarrow M'$ between oriented closed 3-manifolds $M$ and $M'$.

The \emph{cosmetic surgery conjecture} asserts that there are no purely cosmetic surgeries. More generally, the same is conjectured for general 3-manifolds; for a knot $K$ in an oriented 3-manifold $M$ and inequivalent slopes $r,r'$, $M_K(r) \not \cong M_K(r')$.  

Recently, for a knot in $S^{3}$, the Heegaard Floer homology provides quite strong constraints for possibilities of purely cosmetic surgeries \cite{os2,nw,ha}, bringing a great progress toward the affirmative answer to the cosmetic surgery conjecture. Combining other constraints of cosmetic surgeries \cite{iw,it1,de}, the cosmetic surgery conjecture has been confirmed for many cases, such as, knots with at most 17 crossings \cite{de}, composite knots \cite{ta2}, cable knots \cite{ta1}, 2-bridge knots \cite{ijms}, pretzel knots \cite{sz}. In fact, the number of purely cosmetic surgeries, even if exists, is finite in the following sense; for a given $b>0$, there are only finitely many knots that admit a purely cosmetic surgery whose braid index is less than or equal to $b$ \cite{it3}.

On the other hand, for chirally cosmetic surgeries, the situation is more complicated since there are two known families of chirally cosmetic surgeries for knots in $S^{3}$.
\begin{enumerate}
\item[(A)] When $K$ is the $(2,r)$-torus knot, then for each $m \in \Z$,
\[ S^{3}_K\left(\frac{2r^{2}(2m+1)}{r(2m+1)+1}\right) \cong -S^{3}_K\left(\frac{2r^{2}(2m+1)}{r(2m+1)-1} \right).\] 
\item[(B)] When $K$ is amphicheiral, $S^{3}_K(r)\cong - S^{3}_K(-r)$ for all $r$.
\end{enumerate}

A naive and optimistic conjecture is;
\begin{conjecture}
\label{conjecture:main}
A chirally cosmetic surgery on non-trivial knots in $S^{3}$ is either (A) or (B).
\end{conjecture}

This sounds a bit risky because when look at knots in general 3-manifolds, there are more examples of chirally cosmetic surgeries which are different from obvious generalizations of (A) and (B) \cite{bhw,ij}. 

Nevertheless recent researches provide some supporting evidences for the conjecture. In \cite{iis} we discussed several obstructions for chirally cosmetic surgeries and confirmed the conjecture for genus one alternating knots. In \cite{va} the conjecture is proven for alternating odd pretzel knots of genus two or three.
Moreover, in \cite{it2} we showed that a cabled knot, one of a natural candidate admitting chirally cosmetic surgeries other than (A) or (B), does not admit chirally cosmetic surgery as long as the set of the JSJ pieces of the knot exterior does not contain the $(2,r)$-torus knot exterior.

The aim of this paper is to continue discussion for chirally cosmetic surgeries to get further constraints, and demonstrate many knots indeed do not admit chirally cosmetic surgeries.

To study chirally cosmetic surgery, it is useful to separate the following three types;
\begin{description}
\item[$0$-type] $S^3_{K}(r) \cong -S^{3}_K(-r)$.
\item[$+$-type] $S^3_{K}(r) \cong -S^{3}_K(r')$ such that $rr'>0$, 
\item[$-$-type] $S^3_{K}(r) \cong -S^{3}_K(r')$ such that $rr'<0$ and $r+r'\neq 0$.
\end{description}

Then Conjecture \ref{conjecture:main} is divided into the following three conjectures.

\begin{conjecture}
\label{conj:chirally-cosmetic}
Let $K$ be a non-trivial knot in $S^{3}$.
\begin{enumerate}
\item[(i)] $K$ admits a chirally cosmetic surgery of $0$-type if and only if $K$ is amphicheiral.
\item[(ii)] $K$ admits a chirally cosmetic surgery of $+$-type if and only if $K$ is a $(2,p)$-torus knot.
\item[(iii)] $K$ never admits a chirally cosmetic surgery of $-$-type.
\end{enumerate}
\end{conjecture}

In the following, we discuss various constraints for knots to admit chirally cosmetic surgery of the above three types, by using the following invariants;

\begin{itemize}
\item The coefficient $a_{2i}(K)$ of $z^{2i}$ in the Conway polynomial $\nabla_K(z)$ of $K$.
\item $v_3(K)=-\frac{1}{144}V'''_K(1)-\frac{1}{48}V''_K(1) \in \frac{1}{4}\Z$. This is the primitive, canonical degree 3 finite type invariant, normalized so that $v_3(K)=\frac{1}{4}$ for a right-handed trefoil $K$. 
\item A certain canonical degree 5 finite type invariant $v_5(K)$. As we will discuss in Appendix, we can compute $v_5(K)$ from the Kauffman polynomial of $K$ (Theorem \ref{theorem:v5-formula}).
\item the nu-invariant $\nu(K)$ and the tau-invariant $\tau(K)$ of a knot $K$ derived from the Heegaard Floer homology.
\item the determinant $\det(K)$.
\end{itemize}
Also, in the following we denote by $\overline{K}$ the mirror image of $K$.

\subsection{Constraints for chirally cosmetic surgery of $0$-type}

First we discuss constraints for a knot to admit a chirally cosmetic surgery of $0$-type. In \cite{it1} we showed the following.

\begin{theorem}\cite[Corollary 1.3 (ii), Corollary 1.5 (ii)]{it1}
\label{theorem:v3-v5-0-type}
Let $K$ be a knot. If $K$ admits a chirally cosmetic surgery of $0$-type, then $v_3(K)=v_5(K)=0$.
\end{theorem}

Since $v_3(K)=-v_3(\overline{K})$ and $v_5(K)=-v_5(\overline{K})$, this gives a good supporting evidence for conjecture \ref{conj:chirally-cosmetic} (i).

We add a new additional constraint by using the quantum $SO(3)$-invariant. Let $V_K(t)$ be the Jones polynomial of $K$ and let $\tau_5^{SO(3)}$ be the quantum $SO(3)$-invariant at the fifth root of unity $\zeta =\exp(\frac{2\pi \sqrt{-1}}{5})$. 

\begin{theorem}
\label{theorem:quantum-SO3-obstruction}
Let $\zeta = \zeta_5 = \exp(\frac{2\pi \sqrt{-1}}{5})$ be the 5th root of unity.
If $S^{3}_K(\frac{m}{n}) \cong -S^{3}_K(-\frac{m}{n})$ then either 
\begin{enumerate}
\item[(i)] $V_K(\zeta) \in \R$, or,
\item[(ii)] $\tau_5^{SO(3)}(S^{3}_K(\frac{m}{n})) = \tau_5^{SO(3)}(L(m,n))$.
\end{enumerate}
\end{theorem}

Actually, we will give a more general result in Theorem \ref{theorem:quantum-SO3-obstruction-general}. Since $V_{\overline{K}}(t)=V_K(t^{-1})$, if $K$ is amphicheiral then $V_K(\zeta) \in \R$. Thus Theorem \ref{theorem:quantum-SO3-obstruction} provides additional supporting evidence for conjecture \ref{conj:chirally-cosmetic}(i).

Theorem \ref{theorem:quantum-SO3-obstruction} is inspired by a work of Detcherry \cite{de}, where he gave a constraint of purely cosmetic surgery by using the quantum $SO(3)$-invariant. Our results can be seen as a chirally cosmetic surgery analogue of Detcherry's result.

\subsection{Constraints for chirally cosmetic surgery of $\pm$-type}

Next we discuss constraints for a knot to admit a chirally cosmetic surgery of $\pm$-type.

In \cite{it1} we have seen the following constraints by looking at the finite type invariant of degree two.

\begin{theorem}\cite[Corollary 1.3 (iii)]{it1}
\label{theorem:deg-two-criterion}
Let $K$ be a knot. If $K$ admits a chirally cosmetic surgery of $\pm$-type; $S^{3}_K(\frac{m}{n}) \cong -S^{3}_K(\frac{m}{n'})$ with $n+n'\neq 0$, then either (i) or (ii) holds.
\begin{itemize}
\item[(i)] $v_3(K)=0$ and $7a_2(K)^2-a_2(K)-10a_4(K)=0$\footnote{In \cite{it1} we did not write the conclusion $7a_2(K)^{2}-a_2(K)-10a_4(K)=0$, but this follows from the proof therein.}.
\item[(ii)] $v_3(K) \neq 0$ and $\frac{m}{n+n'}= \frac{7a_2(K)^{2}-a_2(K)-10a_4(K)}{8v_3(K)}$.
\end{itemize}
\end{theorem}

This result, combined with the Casson-Walker invariant and the Casson-Gordon invariant, leads the following constraint.

\begin{theorem}\cite[Theorem 6.1]{iis}
\label{theorem:criteria-old}
Let $K$ be a knot in $S^{3}$, and let $d(K)$ be the degree of the Alexander polynomial of $K$. If $K$ admits a chirally cosmetic surgery of $\pm$-type and  $v_3(K)\neq 0$, then\footnote{In \cite{iis} the theorem is stated as non-strict inequality, but its proof actually says that the inequality is strict. Also, in \cite{iis} the necessary assumption $v_3(K)\neq 0$ was omitted (though it was implicit in the arguments).}
\[ 4|a_2(K)| < d(K) \left|\frac{7a_2(K)^2-a_2(K)-10a_4(K)}{8v_3(K)} \right|. \]
\end{theorem}

We add new constraints by combining the using the finite type invariant of degree two and the rank of the Heegaard Floer homology. 

After reviewing and stating the direct consequence from the rank of the Heegaard Floer homology in Proposition \ref{prop:rankHF-result}, We observe that the rank of the Heegaard Floer homology gives the following criterion, which was implicit in \cite{os2}.

\begin{theorem}
\label{theorem:nu}
Let $K$ be a non-trivial knot. If $K$ admits a chirally cosmetic surgery of $\pm$-type then $\max\{\nu(K),\nu({\overline{K}})\} > 0$.
\end{theorem}

We also observe the following remarkable consequence.

\begin{corollary}
\label{cor:no-ZHS}
If $K$ admits a chirally cosmetic surgery $S^{3}_K(\frac{m}{n}) \cong -S^{3}_K(\frac{m}{n'})$, then $|\frac{n+n'}{m}| < 1$. In particular, if $K$ admits a chirally cosmetic surgery of $\pm$-type, then $|m|>2$.
\end{corollary}

\begin{remark}
It is interesting to compare the result with Hanselman's strong constraint for purely cosmetic surgery \cite{ha} (that again comes from Heegaard Floer homology) which says that 
only the remaining possibilities are purely cosmetic surgery of $0$-type (i.e. slopes $r,r'$ satisfy $r+r'=0$) yielding integral homology sphere or homology projective space.
\end{remark}

By combining the rank of the Heegaard Floer homology and Theorem \ref{theorem:deg-two-criterion}, we prove the following.

\begin{theorem}
\label{theorem:weak}
Let $K$ be a knot. If $K$ admits a chirally cosmetic surgery of $\pm$-type and $v_3(K)\neq 0$, then $|7a_2(K)^{2}-a_2(K)-10a_4(K)| > C|4v_3(K)|$, where 
\[ C =
\begin{cases}
4 & \mbox{if } g(K) \neq \max\{\nu(K),\nu({\overline{K}})\}\\
2 & \mbox{otherwise.}
\end{cases}
\]
\end{theorem}

When $K$ is homologically thin (for example, alternating knots), we get a stronger constraint.

\begin{theorem}
\label{theorem:criteria-thin}
Let $K$ be a homologically thin knot. If $K$ admits a chirally cosmetic surgery of $\pm$-type and $v_3(K)\neq 0$, then 
\[ |7a_2(K)^{2}-a_2(K)-10a_4(K)| > \frac{1}{2}(|\det(K)|-2|\tau(K)|-1)|4v_3(K)|\]
\end{theorem}

\subsection{Summary of constraints to admit chirally cosmetic surgery}

We summarize and unify the obstructions for knots to admit chirally cosmetic surgeries obtained so far. 

To state the constraint in an unified form, it is useful introduce the following quantity $O(K)$; for a knot $K$, we define 
\[ O(K)=\begin{cases} 
\left|\dfrac{7a_2(K)^2-a_2(K)-10a_4(K)}{4v_3(K)}\right| & \mbox{ if } v_3(K)\neq 0\\
\infty & \mbox{otherwise}.
\end{cases}
\]

Then arguments discussed so far are summarized in the following form.

\begin{theorem}
\label{theorem:final}
A knot $K$ has no chirally cosmetic surgeries if one of the following conditions is satisfied;
\begin{itemize}
\item[(i-a)] $g(K)=\max\{\nu(K),\nu(\overline{K})\}$, and $O(K) \leq 2$.
\item[(i-a$'$)] $g(K)\neq \max\{\nu(K),\nu(\overline{K})\}$, and $O(K) \leq 4$.
\item[(i-b)] $K$ is homologically thin, and $O(K) \leq \frac{1}{2}(|\det(K)|-2|\tau(K)|-1)$
\item[(i-c)] $O(K) \leq \frac{|8a_2(K)|}{d(K)}$.
\item[(ii)] $v_3(K)\neq 0$, $\max\{\nu(K),\nu({\overline{K}})\} = 0$. 
\item[(iii)] $v_3(K)=0$, $v_5(K)\neq 0$ and $7a_2(K)^2-a_2(K)-10a_4(K) \neq 0$. 
\end{itemize}
\end{theorem}

Here for (i-c), when $d(K)=0$, i.e., the Alexander polynomial of $K$ is trivial, we regard that the condition (i-c) is not satisfied.

The conditions (i-a),(i-a$'$),(i-b),(i-c) states that $K$ has no chirally cosmetic surgeries whenever $O(K)$ is `small' (in particular, they implicitly assume that $v_3(K)\neq 0$ because otherwise we defined $O(K)=\infty$).
Among these four conditions, (i-a)/(i-a$'$) are the weakest, although it can be applied for all knots. In contrast, condition (i-b) is often stronger than the other criteria but it requires homologically thin assumption. Finally, the condition (i-c) is useful only when $d(K)$ is small or $a_2(K)$ is large (typically genus one case).

\subsection{Computations and discussions}

Although it seems to be feasible to use Theorem \ref{theorem:final} to prove Conjecture \ref{conjecture:main} for suitable family of knots (especially, a family of alternating knots), to understand to what extent our current criterion works for  Conjecture \ref{conjecture:main}, we use Theorem \ref{theorem:final} to determine whether a knot admits a chirally cosmetic surgery or not, for knots up to 12 crossings which are neither the $(2,p)$ torus knots nor amphicheiral.

First we check the criterion (i-b) for alternating knots, (i-c), and (iii).
Since $2\tau(K)=\sigma(K)$ holds for alternating knots, to check these three constraints, one can avoid to compute Heegaard-Floer theoretical invariants $\nu(K)$ and $\tau(K)$ which are in general, harder to compute.

The result is summarized in Table \ref{table:result};
\begin{itemize}
\item The column `Target' represents the number of non-amphicheiral, non-$(2,p)$ torus knots.
\item The column `Alternating' represents the number of alternating knots with $v_3(K)\neq 0$ (among the target knots).
\item The column `(i-c)' represents the number of knots such that non-existence of chirally cosmetic surgeries is confirmed by applying the criteria (i-c) in Theorem \ref{theorem:final}. The columns `(i-b)', `(iii)' are similar. 
\end{itemize}

\begin{table}[htbp] \centering

  \begin{tabular}{|c|c|c|c|c|c|c|c|c|}
\hline
  \multicolumn{2}{|c|}{\rm Crossing number} & $\leq 8$ & 9 & 10 & 11a & 11n & 12a & 12n\\ \hline
\multicolumn{2}{|c|}{\rm Target} &25 & 48 & 152 & 366 & 185 & 1234& 884\\ \hline
& Total & 24 &44 & 133 & 324 & 168 &  1095 & 785 \\ \cline{2-9}
 $v_3 \neq 0$ & (i-c) & 3 & 3& 6 &24  & 7 & 54& 28\\  \cline{2-9}
 & Alternating & 21 & 37 & 96 & 324 & - & 1095 & - \\ \cline{2-9}
& (i-b) & 8 & 25 & 68 & 307 & - &1020 & - \\ \hline
$v_3=0$ & Total & 1& 4 & 19 & 42 & 17 & 139 & 99 \\ \cline{2-9}
 & (iii) & 1 & 2 & 18 & 33 & 14 & 109 & 87 \\ \hline
\end{tabular}
\medskip
\caption{Summary of computations}
\label{table:result}
\end{table}

This computation, together with the following additional knowledge/computations
\begin{itemize}
\item our previous result that a non-amphichieral genus one alternating knot, never admits a chirally cosmetic surgery unless it is the trefoil \cite{iis}.
\item Torus knots other than $(2,p)$-torus knots do not admit a chirally cosmetic surgery.
\item information of which knots are quasi-alternating (these knots are homologically thin and $2\tau(K)=\sigma(K)$ so one can apply Theorem \ref{theorem:final} (i-b) like alternating knot case).
\item checking the condition (ii) of Theorem \ref{theorem:final}, where we borrowed an information of nu-invariant from the Knotinfo \cite{knotinfo}. 
\end{itemize}
we conclude the following.

\begin{theorem}
\begin{itemize}
\item[(i)] Up to ten crossings, non-amphicheiral knots do not admit chirally cosmetic surgeries of $0$-type except $9_{28},9_{42},10_{71}$.
\item[(ii)] Up to ten crossings, non-amphicheiral, non-$(2,p)$-torus knots do not admit chirally cosmetic surgeries, except the knots appearing in the following table.
\end{itemize}
\end{theorem}

\begin{table}[h]
  \centering
  \begin{tabular}{|c|c|}
   \hline
   {\rm Crossing number}& {\rm Exceptions} \\
   \hline \hline 
$\leq 8$ & $ 6_2,7_3, 7_5, 8_2, 8_4, 8_5,8_6,8_{11},8_{21}$\\
\hline
  $9$ & $9_3,9_4,9_6,9_9,9_{10},9_{13},9_{25},9_{28},9_{42},9_{43},9_{48},9_{49}$\\
\hline
 & $10_2,10_4,10_6,10_8,10_9,10_{11},10_{16},10_{18},10_{29},$\\
 & $10_{30},10_{39},10_{46},10_{54},10_{61},10_{64},10_{70},10_{71},$\\
$10$& $10_{93},10_{94},10_{106},10_{126},10_{127},10_{128},$\\
& $10_{132},10_{134},10_{138},10_{139},10_{142},10_{144},$ \\
& $10_{145},10_{150},10_{152},10_{154},10_{160},10_{161},10_{162}$\\
\hline
  \end{tabular}
\medskip \medskip
\label{table:result2}
\caption{Knots up to 10 crossings for which non-existence of chirally cosmetic surgery is not detected by our criterion (see Remark \ref{remark:Kegel} for a discussion that they actually do not admit chirally cosmetic surgeries). }
\end{table}

Thus currently our constraints can rule out the existence of chirally cosmetic surgery of $0$-types for all but three non-amphicheiral knots up to 10 crossings, and can rule out the existence of chirally cosmetic surgeries 
for a large portion (roughly 75 \%) of non-amphicheiral, non-$(2,p)$-torus knots up to 10 crossings.

As Table \ref{table:result} shows, when $K$ is alternating the criteria (i-b) will be useful as the crossing number increases. This is because, generically speaking, $\det K$ will grow rapidly than other invariants. In particular, even if Conjecture \ref{conjecture:main} might be false or hard to attack, the following would be more plausible and more tractable;

\begin{conjecture}
\label{conj:alternating}
Non-amphicheiral alternating knots other than the $(2,p)$-torus knots do not admit chirally cosmetic surgery.
\end{conjecture}

On the other hand, for non-alternating knots the criteria (i-c) (or, (i-a)/(i-a$'$)) is less effective. 
Thus it is desirable to find a new constraint which will be effective and useful for non-alternating cases.

\begin{remark}
After completing the paper, we noticed that Varvarezos independently \cite{va2} used the Heegaard Floer homology to give several constraints for knot to admit chirally cosmetic surgeries, some of them essentially overlaps with ours. He used an immersed curve description of the knot Floer homology \cite{hrw}.
\end{remark}

\begin{remark}
\label{remark:Kegel}
After putting the paper on arXiv, M. Kegel \cite{ke} informed us that the knots in Table \ref{table:result2}, that is, the knots for which our criteria cannot check non-existence of chirally cosmetic surgeries, indeed do not admit chirally cosmetic surgeries. Thus Conjecture \ref{conjecture:main} is correct for all the knots up to ten crossings.

He checked this by using the hyperbolic geometry argument \cite[Theorem 1.13]{fps} which we quickly review. 

Assume that $K$ is a non-amphicheiral hyperbolic knot and we denote by $E_K$ the knot exterior. As is discussed in \cite{bhw}, by Thurtson's hyperbolic Dehn surgery theorem, when both the slopes $r$ and $r'$ are sufficiently large, then both $S^{3}_K(r)$ and $S^{3}_K(r')$ are hyperbolic, and the core circles $c_r$, $c_{r'}$ of the attached solid tori are isotopic to the unique shortest geodesic. Therefore by the Mostow rigidity, if there is an orientation-preserving homeomorphism $f: S^{3}_K(r) \rightarrow -S^{3}_K(r')$, then the homeomorphism $f$ sends $c_{r}$ to $c_{r'}$. Since we are assuming that $K$ is not amphichiral, this means that the slopes $r$ and $r'$ are equivalent hence $r=r'$.

The `sufficiently large slope' condition was made more explicit in \cite[Theorem 7.29]{fps}. Let
\[ \mathcal{S}_1 = \left\{ r \: \middle| \: \frac{\ell(r)}{\sqrt{\mbox{area}\,\partial E_K}} < \max \left( 10.1, \sqrt{\frac{2\pi}{sys(E_K)}+58} \right) \right\}.\]
Here we view a slope $r$ as a geodesic on the Euclidian torus $\partial E_K$ and denote by $\ell(r)$ the Euclidian length of $r$. If $K$ admits a chirally cosmetic surgery $S^{3}_K(r) \cong -S^{3}_K(r')$ then we may assume that $r \in \mathcal{S}_1$. 

By looking at the volume (more precisely, the lower bound of the volume of Dehn-fillings \cite[Theorem 1.1]{fkp}), it follows that the other slope $r'$ is also contained in a finite set
\[ \mathcal{S}_2 = \left\{ r \: \middle| \: \ell(r) \leq 2\pi \left(1- \left(\frac{V}{vol(E_K)} \right)\right) \right\} \]
where $V=\max \{vol(S^{3}_K(r)) \: | \: r \in \mathcal{S}_1 \}$.
Therefore we have the following.

\begin{theorem}\cite[Theorem 1.13]{fps}
\label{theorem:fps-argument}
If a non-amphicheiral hyperbolic knot $K$ admits a chirally cosmetic surgery  $S^{3}_K(r) \cong -S^{3}_K(r')$, then $(r,r') \in \mathcal{S}_1 \times \mathcal{S}_2$.
\end{theorem}

Since the sets $\mathcal{S}_1$ and $\mathcal{S}_2$ are finite and explicitly computable, by checking all the candidates of chirally cosmetic surgery one can check whether a given (non-amphicheiral) hyperbolic knot admit a chirally cosmetic surgery or not.

However, we remark that to use Theorem \ref{theorem:fps-argument}, we need to be a bit careful; computations of knot invariants appeared in Theorem \ref{theorem:final} is combinatorial so their computations are rigid, whereas the computations to use Theorem \ref{theorem:fps-argument} is numerical. Thus to make the argument rigorous, we require more careful treatments (see \cite{hikmot} for details). 

Thus Theorem \ref{theorem:final} has its own merits and it deserves to explore further constraints to exclude the knots in Table \ref{table:result2}, although they actually do not admit chirally cosmetic surgeries.
\end{remark}

\section*{Acknowledgement}
The first author is partially supported by JSPS KAKENHI
Grant Number 18K03287. The second author is partially supported by JSPS KAKENHI Grant Numbers 19K03490, 21H04428. The third author is partially supported by JSPS KAKENHI Grant Number 21K03244.
We would like to thank M. Kegel for valuable comments and his computations on non-existence of chirally cosmetic surgeries based on hyperbolic geometry computations.

\section{Constraint for chirally cosmetic surgery from quantum $SO(3)$-invariants}

\subsection{Quantum $SO(3)$-invariant}

First of all we quickly review the definition and basic properties of the quantum 
$SO(3)$-invariant of 3-manifolds. For details, we refer to \cite{oh}.

Although the quantum $SO(3)$-invariant has been actively studied and various deeper properties are known, we only use fundamental properties which was proved in \cite{km}\footnote{In \cite{km} the quantum $SO(3)$-invariant $\tau_r^{SO(3)}$ is denoted by $\tau'_r$.}, where the quantum $SO(3)$-invariant was first defined. 

For $n \in \Z$, let $[n] = \frac{q^{\frac{n}{2}}- q^{-\frac{n}{2}}}{q^{\frac{1}{2}}- q^{-\frac{1}{2}}}$ the quantum integer.

For a framed $m$-component oriented link $L=K_1\cup \cdots \cup K_m$, let $Q^{\mathfrak{sl}_2;V_{n_1},\ldots,V_{n_m}}(L) \in \Z[q^{\frac{1}{2}}]$ be the quantum $\mathfrak{sl}_2$ invariant of $L$, where the $i$-th component $L_i$ is colored by the $n_i$-dimensional irreducible representation $V_{n_i}$ of $U_q(\mathfrak{sl}_2)$.

Let $a_i$ be the framing of the $i$-th component $K_i$ of $L$, and let $L_0$ be the same link $L$ with zero framings. Then
\begin{equation}
\label{eqn:framing} Q^{\mathfrak{sl}_2;V_{n_1},\ldots,V_{n_m}}(L) = q^{\sum_{i=1}^{m}\frac{n_i^2 -1}{4}a_{i}} Q^{\mathfrak{sl}_2;V_{n_1},\ldots,V_{n_m}}(L_0) 
\end{equation}

Let $M$ be a 3-manifold obtained by a Dehn surgery on a framed link $L$.
For odd $r$, the quantum $SO(3)$ invariant $\tau^{SO(3)}_r(M) \in \C$ is defined as follows:
\begin{equation}
\label{eqn:def-SO3} \tau_r^{SO(3)}(M) = c_{+}^{-\sigma_+}c_{-}^{-\sigma_-}\sum_{\substack{1\leq n_1,\ldots, n_m < r \\ n_i: odd}}[n_1]\cdots[n_m]Q^{\mathfrak{sl}_2;V_{n_1},\ldots,V_{n_m}}(L)|_{q=\zeta}
\end{equation}
Here
\begin{itemize}
\item $\zeta = \zeta_r= \exp(\frac{2\pi \sqrt{-1}}{r})$ is the $r$-th root of unity.
\item $\displaystyle c_{\pm} =\sum_{\substack{1\leq n < r \\ n: odd}}[n]Q^{\mathfrak{sl}_2;V_{n}}(U_{\pm})|_{q=\zeta}$, where $U_{\pm}$ is the $(\pm 1)$-framed unknot.
\item $\sigma_{+}$ (resp. $\sigma_{-}$) is the number of positive (resp. negative) eigenvalues of the linking matrix of $L$.
\end{itemize}

From the definition, one can see that the quantum $SO(3)$-invariant $\tau_r^{SO(3)}$ satisfies the following properties.
\begin{itemize}
\item $\tau_r^{SO(3)}(S^{3})=1$\\
\item $\tau_r^{SO(3)}(-M)=\overline{\tau_r^{SO(3)}(M)}$\\
\end{itemize}

\subsection{Quantum $SO(3)$ invariant for rational surgery on knots}

 In this section we give a (schematic) formula of the quantum $SO(3)$ invariant for a 3-manifold obtained by a rational surgery on a knot $K$ in $S^{3}$ (Proposition \ref{prop:SO3-linear-sum}), which is well-known for experts.

For a link $L$, let $L'$ be the link obtained by adding the meridian of one of its component $L_i$ of $L$. Then the quantum $\mathfrak{sl}_2$ invariants at the $r$-th root of unity of $L'$ (so that the added meridian is colored by $V_k$ and the component $L_i$ is colored by $V_j$) and $L$ are related as follows \cite[(3.27) Lemma]{km}.
\begin{equation}
\label{eqn:formula}
 Q^{\mathfrak{sl}_2}\left.\left(\raisebox{-3mm}{
\begin{picture}(40,30)
\put(0,-4){$V_k$}
\put(25,24){$V_j$}
\put(13,12){\oval(16,16)[l]}
\put(13,4){\line(1,0){14}}
\put(20,-6){\line(0,1){6}}
\put(20,10){\line(0,1){22}}
\put(27,12){\oval(16,16)[r]}
\end{picture} } 
\right)\right|_{q=\zeta}
= \frac{[jk]}{[j]}
 Q^{\mathfrak{sl_2}}\left.\left(\raisebox{-3mm}{
\begin{picture}(40,30)
\put(25,24){$V_j$}
\put(20,-6){\line(0,1){38}}
\end{picture} } 
\right)\right|_{q=\zeta}
\end{equation}

This allows us to write the quantum $SO(3)$-invariant of 3-manifolds obtained by a \emph{rational} surgery on a knot in $S^{3}$. For a rational number $\frac{m}{n}$, take its continued fraction expansion $[a_0,\ldots,a_{\ell}]$ such as
\[ \frac{m}{n} = a_0-\cfrac{1}{a_1-\cfrac{1}{a_2 - \cfrac{1}{\ddots -\cfrac{1}{a_\ell}}}}.\]
Then the $\frac{m}{n}$-surgery on a knot $K$ is expressed as an integral surgery on the framed link $L$ that consists of $K$ and the Hopf chain (Fig \ref{fig:rational}),

\begin{figure}[htbp]
\begin{center}

\includegraphics*[width=60mm]{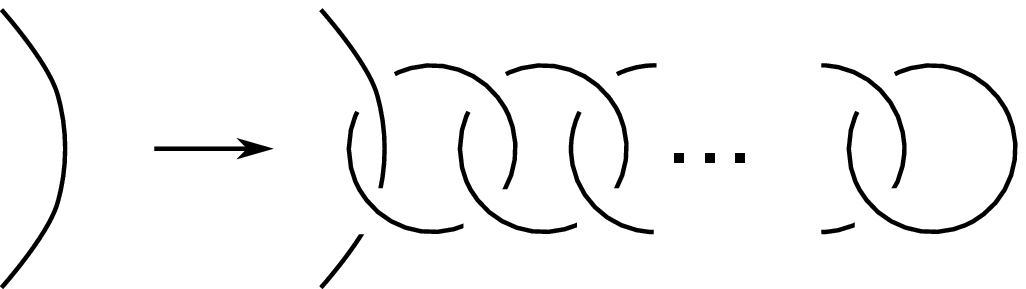}
\begin{picture}(0,0)
\put(-165,-5) {$K$}
\put(-190,45) {$\frac{m}{n}$}
\put(-115,-5) {$L$}
\put(-130,45) {$a_0$}
\put(-105,42) {$a_1$}
\put(-90,42) {$a_2$}
\put(-20,42) {$a_\ell$}
\end{picture}
\caption{Rational surgery on $K$ as integral surgery on Hopf chain $L$} 
\label{fig:rational}
\end{center}
\end{figure} 

Thus thanks to the formula (\ref{eqn:framing}) and (\ref{eqn:formula}), for the link $L$ in Figure \ref{fig:rational}, its quantum $\mathfrak{sl}_2$ invariant 
$Q^{\mathfrak{sl}_2; V_{n_0},\ldots,V_{n_{\ell}}}(L)|_{q=\zeta_r}$ is actually written by $n_{1},\ldots,n_{\ell}$, $a_1,\ldots,a_{\ell}$ and $Q^{\mathfrak{sl}_2; V_{n_0}}(K_0)|_{q=\zeta}$.

In the definition (\ref{eqn:def-SO3}) of the quantum $SO(3)$ invariant, the summation runs over odd integers $n_1,\ldots,n_{m}$ satisfying $1\leq n_i <r$.
However, it is known that when we change the color $n_m$ of the $m$-th component of the link $L$ to $(r-n_m)$, then
\[ Q^{\mathfrak{sl}_2;V_{n_1},\ldots,V_{n_{m-1}},V_{r-n_m}}(L)|_{q=\zeta_r} = (\sqrt{-1})^{(r-2n_{m})a+2\lambda} Q^{\mathfrak{sl}_2;V_{n_1},\ldots,V_{n_{m-1}},V_{n_m}}(L)|_{q=\zeta_r}\]
holds \cite[(4.20) Symmetry principle]{km}.
Here $a$ is the framing of the $m$-th component, and $\displaystyle \lambda = \sum_{n_i\textrm{ is even}}lk(L_i,L_
m)$.
In particular, for the $0$-framed knot $K_0$
\begin{equation}
\label{eqn:s-principle}
Q^{\mathfrak{sl}_2;V_{r-n}}(K_0)|_{q=\zeta_r} = Q^{\mathfrak{sl_2};V_{n}}(K_0)|_{q=\zeta_r}. 
\end{equation} 

Thus summarizing the argument so far we get the following, which seems to be well-known and appeared elsewhere, although we do not know a reference.

\begin{proposition}
\label{prop:SO3-linear-sum}
$\tau_r^{SO(3)}\left(S^{3}_K(\frac{m}{n}) \right)$ is expressed as a linear combination of the quantum invariant of the $0$-framed knot $K_0$, $\{Q^{\mathfrak{sl}_2;V_i}(K_0)|_{q=\zeta_r}\: | \: i=1,2,\ldots, \frac{r-1}{2}\}$. Namely, there are constant $c^{i}_r(\frac{m}{n})$ that only depends on the surgery coefficient $\frac{m}{n}$ such that 
\[ \tau_{r}^{SO(3)}\left(S^{3}_K(\frac{m}{n}) \right) =\sum_{1\leq i< \frac{r}{2} }  \left.c^{i}_r\left(\frac{m}{n} \right)  Q^{\mathfrak{sl}_2;V_{i}}(K_0)\right|_{q=\zeta_r} \] 
Moreover, $c^i_r\left(\frac{m}{n}\right) = \overline{c^{i}_r\left(-\frac{m}{n}\right)}$.
\end{proposition}
\begin{proof}
Let $[a_{0},\ldots,a_{\ell}]$ be a continued fraction expansion of $\frac{m}{n}$, and let $L$ be the corresponding (integer-framed) link. 

Let $\sigma_\pm(L)$ be the number of positive and negative eignvalues of the linking matrix of $L$. Then

\begin{align*}
&\tau_r^{SO(3)} \left( S^{3}_K(\frac{m}{n}) \right) = \tau_r^{SO(3)}(S^{3}_L)\\
&\quad = c_+^{-\sigma_+(L)}c_-^{-\sigma_-(L)}
\sum_{\substack{1\leq n_0,\ldots, n_\ell < r \\ n_0,\ldots, n_{\ell}: odd}}[n_0] \cdots [n_\ell]Q^{\mathfrak{sl}_2;V_{n_0},\ldots,V_{n_\ell}}(L)|_{q=\zeta_r}\\
&\quad \stackrel{(\ref{eqn:framing})}{=} c_+^{-\sigma_+(L)}c_-^{-\sigma_-(L)}
\sum_{\substack{1\leq n_0,\ldots, n_\ell < r \\ n_0,\ldots, n_{\ell}: odd}} q^{\sum_{i=0}^{\ell}\frac{n_i^2-1}{4}a_{i}}[n_0] \cdots [n_\ell]Q^{\mathfrak{sl}_2;V_{n_0},\ldots,V_{n_\ell}}(L_0)|_{q=\zeta_r} \\
&\quad \stackrel{(\ref{eqn:formula})}{=} c_+^{-\sigma_+(L)}c_-^{-\sigma_-(L)} \\
&\quad \qquad\sum_{\substack{1\leq n_0,\ldots, n_\ell < r \\ n_0,\ldots, n_{\ell}: odd}} q^{\sum_{i=0}^{\ell}\frac{n_i^2-1}{4}a_{i}} [n_0n_1][n_1n_2]\cdots[n_{\ell-1}n_{\ell}][n_{\ell}]Q^{\mathfrak{sl}_2;V_{n_0}}(K_0)|_{q=\zeta_r}.
\end{align*}

We put $n(i)=
\begin{cases}
i & (i \mbox{: odd}) \\
r-i & (i \mbox{: even})
\end{cases}$.
Using (\ref{eqn:s-principle}), we rewrite the summation over $n_0=1,3,\ldots,r$ in terms of $i=1,2,\ldots \frac{r-1}{2}$.
\begin{align*}
&\tau_r^{SO(3)} \left( S^{3}_K(\frac{m}{n})\right) \\
& = c_+^{-\sigma_+(L)}c_-^{-\sigma_-(L)}\\
& \quad
\sum_{i=1}^{\frac{r-1}{2}}\left( q^{\frac{n(i)^2-1}{4}a_0}
\sum_{\substack{1\leq n_1,\ldots, n_\ell < r \\ n_1,\ldots, n_{\ell}: odd}}q^{\sum_{i=1}^{\ell}\frac{n_i^2-1}{4}a_{i}} [n(i)n_1][n_1n_2]\cdots[n_{\ell-1}n_{\ell}][n_{\ell}]\right) Q^{\mathfrak{sl}_2;V_{i}}(K_0)|_{q=\zeta_r}
\end{align*}
So we conclude
\[ \tau_{r}^{SO(3)}(S^{3}_K(\frac{m}{n})) = \sum_{1\leq i \leq \frac{r-1}{2} } c^{i}_r(\frac{m}{n})Q^{\mathfrak{sl}_2;V_{i}}(K_0)|_{q=\zeta_r}\] 
where
\begin{equation}
\label{formula:coefficient} c_{r}^{i}(\frac{m}{n}) = c_+^{-\sigma_+(L)}c_-^{-\sigma_-(L)}q^{\frac{n(i)^2-1}{4}a_0} \sum_{\substack{1\leq n_1,\ldots, n_\ell < r \\ n_1,\ldots, n_{\ell}: odd}} q^{\sum_{i=1}^{\ell}\frac{n_i^2-1}{4}a_{i}} [n(i)n_1][n_1n_2]\cdots[n_{\ell-1}n_{\ell}][n_{\ell}].
\end{equation}

To compare $ c_{r}^{i}(\frac{m}{n})$ and $c^{i}_{r}(-\frac{m}{n})$ we take $[-a_{0},\ldots,-a_{\ell}]$ as a continued fraction expansion of $-\frac{m}{n}$. Let $L'$ be the corresponding framed link. Then by (\ref{formula:coefficient})
\[ c^{i}_{r}(-\frac{m}{n}) = c_+^{-\sigma_+(L')}c_-^{-\sigma_-(L')}q^{\frac{n(i)^2-1}{4}(-a_0)}\sum_{\substack{1\leq n_1,\ldots, n_\ell < r \\ n_1,\ldots, n_{\ell}: odd}} q^{\sum_{i=1}^{\ell}\frac{n_i^2-1}{4}(-a_{i})} [n(i)n_1][n_1n_2]\cdots[n_{\ell-1}n_{\ell}][n_{\ell}] \] 

Since $\overline{c_\pm}=c_\mp$ and $\sigma_\pm(L) = \sigma_{\mp}(L')$ we conclude

\begin{align*}
\overline{c^{i}_{r}(-\frac{m}{n})} &= c_+^{-\sigma_+(L)}c_-^{-\sigma_-(L)}q^{\frac{n(i)^2-1}{4}a_0}\sum_{\substack{1\leq n_1,\ldots, n_\ell < r \\ n_1,\ldots, n_{\ell}: odd}} q^{\sum_{i=1}^{\ell}\frac{n_i^2-1}{4}a_{i}} [n(i)n_1][n_1n_2]\cdots[n_{\ell-1}n_{\ell}][n_{\ell}]\\
& = c^{i}_{r}(\frac{m}{n}) 
\end{align*}
as desired.
\end{proof}

\subsection{Constraint for chirally cosmetic surgery of $0$-type from }

Proposition \ref{prop:SO3-linear-sum} leads to the following criterion for chirally cosmetic surgery of $0$-type.
For odd $r$, let $V_r(K) = \begin{pmatrix} Q^{\mathfrak{sl}_2;V_1}(K_0)|_{q=\zeta_r} \\  Q^{\mathfrak{sl}_2;V_2}(K_0)|_{q=\zeta_r} \\ \vdots \\ 
Q^{\mathfrak{sl}_2;V_{\frac{r-1}{2}}}(K_0)|_{q=\zeta_r}\end{pmatrix}$ and 
$v(r,\frac{m}{n})=\begin{pmatrix} c^{1}_r(\frac{m}{n}) \\ c^{2}_r(\frac{m}{n}) \\\vdots \\c^{\frac{r-1}{2}}_r(\frac{m}{n}) \end{pmatrix}$. Proposition \ref{prop:SO3-linear-sum} says that $\tau^{SO(3)}_r(S^{3}_K(\frac{m}{n}))=V_r(K) \cdot v(r,\frac{m}{n})$.

\begin{theorem}
\label{theorem:quantum-SO3-obstruction-general}
If $S^{3}_K(\frac{m}{n}) \cong -S^{3}_K(-\frac{m}{n})$, then for each odd $r$
\[ V_r(K)\cdot v(r,\tfrac{m}{n}) = \overline{V_r(K)}\cdot v(r,\tfrac{m}{n}). \]
\end{theorem}
\begin{proof}

Since we have seen that $\overline{c^{i}_r(-\frac{m}{n})}= c^{i}_r(\frac{m}{n})$,  $v(r,\frac{m}{n}) = \overline{v(r,-\frac{m}{n})}$. Thus if $S^{3}_K(\frac{m}{n}) \cong -S^{3}_K(-\frac{m}{n})$ 
\begin{align*}
V_r(K) \cdot v(r,\tfrac{m}{n}) & = \tau_r^{SO(3)}\left(S^{3}_K(\frac{m}{n}) \right) \\
& = \tau_r^{SO(3)}\left(-S^{3}_K(-\frac{m}{n})\right)\\
& = \overline{V_r(K) \cdot v(r,-\tfrac{m}{n})} \\
&=  \overline{V_r(K)} \cdot v(r,\tfrac{m}{n})
\end{align*}
\end{proof}

Theorem \ref{theorem:quantum-SO3-obstruction} is a special case of Theorem \ref{theorem:quantum-SO3-obstruction-general} for $r=5$. Recall that $Q^{\mathfrak{sl}_2;V_1}(K_0)=1$ and $Q^{\mathfrak{sl}_2;V_2}(K_0)=[2]V_K(t)|_{t=q^{-1}}$, where $V_K(t)$ is the Jones polynomial of $K$.

\begin{proof}[Proof of Theorem \ref{theorem:quantum-SO3-obstruction}]
By Theorem \ref{theorem:quantum-SO3-obstruction-general}, if $S^{3}_K(\frac{m}{n}) \cong -S^{3}_K(-\frac{m}{n})$, then $V_K(\zeta_5)c^{2}_5(\frac{m}{n}) = \overline{V_K(\zeta_5)}c^{2}_5(\frac{m}{n})$. 
If $c^{2}_5(\frac{m}{n}) =0$, then $\tau^{SO(3)}_5(S^{3}_K(\frac{m}{n}))$ does not depend on $K$ so  
\[ \tau^{SO(3)}_5\left(S^{3}_K(\frac{m}{n}) \right) = \tau^{SO(3)}_5\left(S^{3}_{\sf Unknot}(\frac{m}{n})\right) = \tau^{SO(3)}_5(L(m,n)).\]
Thus when  $\tau^{SO(3)}_5(S^{3}_K(\frac{m}{n})) \neq \tau^{SO(3)}_5(L(m,n))$, then $c^{2}_5(\frac{m}{n}) \neq 0$. Hence $V_{K}(\zeta_5) \in \R$.
\end{proof}

\section{Rank of the Heegaard Floer homology}
\label{section:HF}
We review the rational surgery formula of $\textrm{rank}\,\widehat{HF}$, which is a consequence of rational surgery formula of the knot floer homology.

Throughout the section, by taking a mirror image if necessary, we always assume that $\nu(K) \geq \nu(\overline{K})$ and $\nu(K)\geq 0$ hold. 
Since $S^{3}_{\overline{K}}(r) \cong -S^{3}_K(-r)$, this does not affect non-existence of chirally cosmetic surgeries.

\subsection{Surgery formula of $\textrm{rank}\,\widehat{HF}$}

We quickly review the Heegaard Floer homology theory which is needed to describe the rational surgery formula of $\textrm{rank}\,\widehat{HF}$. We follow a formal point of view as discussed in \cite{hw}, and we work on the coefficient $\mathbb{F}=\Z\slash 2\Z$.

In the following, we often ignore various aspects or structures of knot floer homologies which are irrelevant in our applications. 

For a knot $K$, a theory of Heegaard Floer homology provides a graded, bi-filtered chain complex called the \emph{(infinity) knot Floer chain complex} $CFK^{\infty}(S^{3},K)$, which is unique up to filtered chain homotopy equivalence.
As a vector space over $\mathbb{F}$, $CFK^{\infty}(S^{3},K)$ is generated by a basis $\mathcal{B}$. We denote by $\mathcal{F}(\mathbf{x})=(i,j)$ the filtration of an element $\mathbf{x} \in \mathcal{B}$. Moreover, $CFK^{\infty}(S^{3},K)$ has a structure of finitely generated free module over $\mathbb{F}[U,U^{-1}]$. There is a finite subset $\mathcal{G}$ of $\mathcal{B}$ such that $\mathcal{B}=\{U^d \mathbf{x} \: | \: d\in \Z, \mathbf{x} \in \mathcal{G} \}$. Multiplying $U$ decreases the filtration by $(-1,-1)$, namely, $\mathcal{F}(U\mathbf{x}) = \mathcal{F}(\mathbf{x})+(-1,-1)$.

For $s\in \Z$, let 
\[ \widehat{A}_{s} = C\{\max\{i,j-s\}=0\} \]
be the complex generated by $\{\mathbf{x} \in \mathcal{B} \: | \: \mathcal{F}(\mathbf{x}) \in \{(i,j) \: | \: \max\{i,j-s\}=0\}\}$.
Similarly, let
\[ \widehat{B}=C\{i=0\}.\]
be the complex generated by $\{\mathbf{x} \in \mathcal{B} \: | \: \mathcal{F}(\mathbf{x}) \in \{\{(i,j) \: | \: i=0\}\}$. It is known that $\widehat{B}$ is chain homotopic to the Heegaard Floer chain complex of $S^{3}$, so in particular, $H_*(\widehat{B})\cong \widehat{HF}(S^{3})\cong \mathbb{F}$.
Let 
\[ \widehat{v}_k: \widehat{A}_{s} \rightarrow \widehat{B}\]
be the projection map. The nu-invariant is defined by 
\[ \nu(K)= \min\{s \in \Z\: | \: (\widehat{v}_k)_*:H_*(\widehat{A}_s) \rightarrow H_*(\widehat{B})(\cong \mathbb{F}) \mbox{ is non-trivial} \}.\]  
Finally, let
\[ C_K=\sum_{s \in \Z} \textrm{rank}(H_*(\widehat{A}_s)-1).\]
It is known that $C_K$ is always an even non-negative integer.

Using these quantities, (and under the setting that $\nu(K) \geq \nu(\overline{K})$ and $\nu(K) \geq 0$), the rational surgery formula of $\widehat{HF}$ is given as follows\footnote{In \cite{os2} the slope (rational number) $\frac{m}{n}$ is expressed so that $n>0$. Here we rewrite the formula in terms of our convention that $\frac{m}{n}$ is expressed so that $m>0$.}.

\begin{proposition}\cite[Proposition 9.6]{os2}
\label{proposition:HFrank}
\[ \textrm{rank}\,\widehat{HF}(S^{3}_K(\frac{m}{n})) =
\begin{cases} \displaystyle
m + nC_K& (\frac{m}{n} \geq 2\nu(K)-1))\\
-m + (4\nu(K)-2)n + nC_K & (0 \leq \frac{m}{n} \leq 2\nu(K)-1)\\
m - (4\nu(K)-2)n - nC_K & (n<0, \nu(K)>0)\\
m - nC_K & (n<0, \nu(K)=0)
\end{cases}
\]
\end{proposition}

\subsection{Computation of $C_K$}
\label{section:C_K}

In our purpose, to apply the formula of $\textrm{rank}\,\widehat{HF}$ it is useful to estimate or compute $C_K$.
To begin with, we observe that $C_K\geq 2$ if $K$ is non-trivial (i.e., $g(K)\neq 0$) and $g(K)\neq \nu(K)$; this follows from 
\cite[Proposition 9.7]{os2}, which asserts
\begin{equation}
\label{eqn:genus} g(K)=\max\{\nu(K),\{s \in \Z \: | \: \textrm{rank}\,H_*(\widehat{A}_{s-1})>1 \}\}.
\end{equation}

Slightly extending this observation, we get the following. 
\begin{lemma}
\label{lemma:C_K}
If $g(K)\neq 1$ and $\nu(K) \neq g(K)$, then $C_K\geq 4$.
\end{lemma}
\begin{proof}
The complex $\widehat{A}_s$ has the following properties \cite{os2}.
\begin{itemize}
\item $\chi(\widehat{A}_s)=1$
\item $\textrm{rank}H_*(\widehat{A}_s) =\textrm{rank}H_*(\widehat{A}_{-s})$
\end{itemize}
So when $C_K=2$ then $\textrm{rank}\,H_{*}(\widehat{A}_s)=\begin{cases} 3 & (s=0) \\ 1 & (otherwise) \end{cases}$. Thus by (\ref{eqn:genus}) either $g(K)=1$ or $\nu(K)\neq g(K)$.
\end{proof}

Let $\Delta_{K}(t)=d_0+\sum_{i=1}^{k}d_i(t^{-i}+t^{i})$ be the Alexander polynomial, normalized so that $\Delta_K(t)=\Delta_K(t^{-1})$ and $\Delta(1)=1$.
A knot $K$ is \emph{(Floer) homologically thin}, if there is a constant $s$ such that the knot Floer homology $\widehat{HFK}_{i}(K,j)$ is non-trivial only if $j-i=s$. Alternating knots \cite{os1}, or more generally, quasi-alternating knots \cite{mo} are typical classes of homologically thin knots.  

If $K$ is homologically thin, then the knot Floer chain complex $CFK^{\infty}(S^{3},K)$ is determined by $\tau(K)$ and $\Delta_K(t)$ \cite{os1,pe}. This allows us to deduce a formula of $C_K$ in the rank formula.

\begin{lemma}
\label{lemma:thin-C_K}
Assume that $K$ is homologically thin. Then
\[ C_K= \frac{1}{2}(|\det(K)|-2|\tau(K)|-1)\]
\end{lemma}
\begin{proof}

Since $K$ is homologically thin, the signs of the coefficients of $\Delta_K(t)$ is alternating. Hence
\[ |\det(K)| = |\Delta_K(-1)|= 2(|d_0|+|d_1|+\cdots + |d_{g}|)-|d_0|.\]  
The knot Floer chain complex $CFK^{\infty}(S^{3},K)$ of a homologically thin knot $K$ is chain homotopy equivalent to a direct sum of the model complexes called the \emph{square} and the \emph{staircase} of height $\tau(K)$ depicted in Figure \ref{fig:model-complex}, and given as follows \cite{pe}. (Here, by the staircase of height $h$ we mean the staircase-shaped chain complex having $2h+1$ generators). See Example \ref{example:HFcomplex}.

\begin{figure}[htbp]
\begin{center}
\includegraphics*[width=60mm]{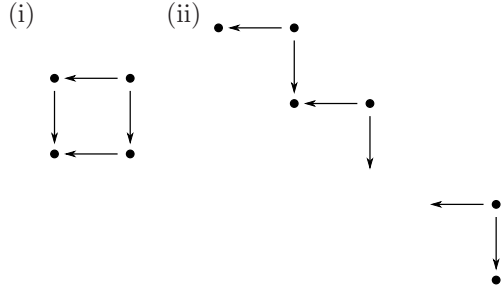}
\begin{picture}(0,0)
\put(-190,100) {(i)}
\put(-130,100) {(ii)}
\end{picture}
\caption{Model complex: (i) square, (ii) staircase} 
\label{fig:model-complex}
\end{center}
\end{figure} 
 For $i=0,1,\ldots,$ let $\varepsilon_i=\begin{cases}1 & i\leq \tau(K) \\ 0 & i>\tau(K)\end{cases}$, and let $(\delta_1,\ldots,\delta_g)$ be non-negative integers defined by the equations 
\begin{align*}
|d_g|&=\varepsilon_{g}+ \delta_{g} \\
|d_{g-1}|&=\varepsilon_{g-1}+\delta_{g-1}+2\delta_{g}\\
|d_{g-2}|&=\varepsilon_{g-2}+\delta_{g-2}+2\delta_{g-1}+\delta_{g}\\
&\vdots \\
|d_{2}| &= \varepsilon_{2}+\delta_{2}+2\delta_{3}+\delta_4 \\
|d_{1}|&= \varepsilon_{1} +\delta_{1}+2\delta_{2}+\delta_3\\
|d_{0}| &= \varepsilon_{0}+ 2\delta_{1}+2\delta_{2}
\end{align*}
(When $g=1$, the last equation is understood as $d_0=2\delta_1+\varepsilon_0$). 
$CFK^{\infty}(S^{3},K)$ is filtered chain homotopy equivalent to the complex depicted in Figure \ref{fig:complex}.

\begin{figure}[htbp]
\begin{center}
\includegraphics*[width=85mm]{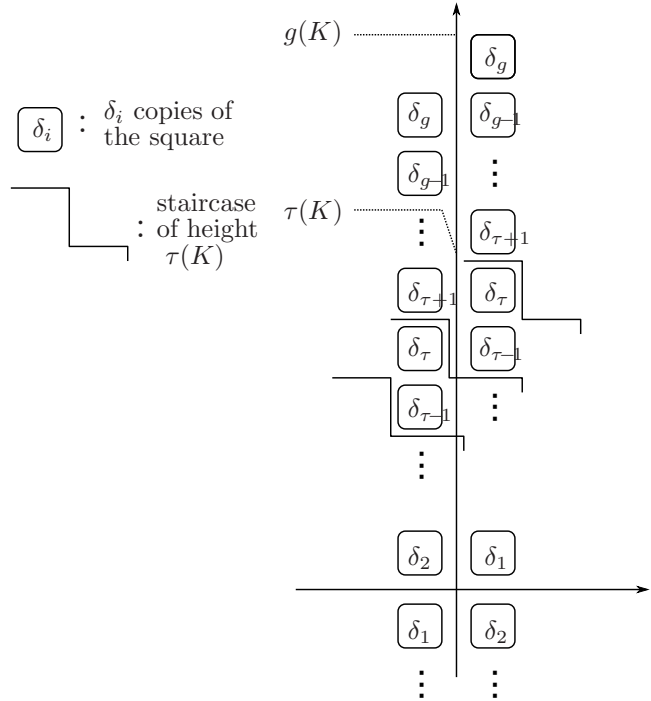}
\begin{picture}(0,0)
\put(-220,215) {\LARGE :}
\put(-237,213) {$\delta_i$}
\put(-210,220) {$\delta_i$ copies of}
\put(-210,210) {the square}
\put(-198,175) {\LARGE :}
\put(-190,185) {staircase}
\put(-190,175) {of height}
\put(-190,165) { $\tau(K)$}
\put(-142,182) {$\tau(K)$}
\put(-142,250) {$g(K)$}
\put(-66,240) {$\delta_g$}
\put(-68,219) {$\delta_{g\!-\!1}$}
\put(-95,219) {$\delta_{g}$}
\put(-95,196) {$\delta_{g\!-\!1}$}
\put(-68,174) {$\delta_{\tau+1}$}
\put(-66,151) {$\delta_{\tau}$}
\put(-95,151) {$\delta_{\tau+1}$}
\put(-95,130) {$\delta_{\tau}$}
\put(-68,130) {$\delta_{\tau\!-\!1}$}
\put(-95,107) {$\delta_{\tau\!-\!1}$}
\put(-66,51) {$\delta_{1}$}
\put(-95,51) {$\delta_{2}$}
\put(-66,23) {$\delta_{2}$}
\put(-95,23) {$\delta_{1}$}
\end{picture}
\caption{$CFK^{\infty}(S^{3},K)$ for a homologically thin knot $K$: Here we only write a complex near the filrataion $(0,*)$.} 
\label{fig:complex}
\end{center}
\end{figure} 

Thus $\textrm{rank}\, H_*(\widehat{A}_s) = 2\delta_{s+1}+1$ and
\begin{align*}
 C_K &= \sum_{s\in \Z}(\textrm{rank}H_*(\widehat{A}_s)-1)= 4(\delta_{g}+\delta_{g-1}+\cdots + \delta_{2}+\delta_1)- 2\delta_{1}\\
\end{align*}

By the defining equations of $\delta_1,\ldots, \delta_g$, we get
\begin{align*}
|d_g|+|d_{g-1}|+\cdots+|d_0| & =(\varepsilon_g+\cdots+\varepsilon_{0}) + 4(\delta_{g}+\cdots+\delta_{1})+ (\delta_2-\delta_1)  \\
& =(\tau+1)+ 4(\delta_{g}+\cdots+\delta_{1}) +(\delta_2-\delta_1)
\end{align*}
Hence
\begin{align*}
 C_K &= (|d_g|+\cdots +|d_0|) - \tau(K) -1  -(\delta_2-\delta_1) -2\delta_1\\
&= (|d_g|+\cdots +|d_0|) - \tau(K) -1  -(\delta_1+\delta_2)\\
&= \frac{1}{2}(|\det(K)|+|d_0|) -\tau(K) -1 -\frac{1}{2}(|d_0|-1)\\
&= \frac{1}{2}(|\det(K)|-2|\tau(K)|-1)
\end{align*}

\end{proof}

\begin{example}
\label{example:HFcomplex}
Let $K$ be (the mirror image of) the knot $6_2$, whose Alexander polynomial is $t^{-2}-3t^{-1}+3-3t+t^2$ and $\tau(K)=1$. Thus $\delta_2=1, \delta_1=0$, so the knot Floer chain complex $CFK^{\infty}(S^{3},K)$ is given as Figure \ref{fig:complex-example}.  The shaded region represents the complex $\widehat{A}_1$. 

\begin{figure}[htbp]
\begin{center}
\includegraphics*[width=60mm]{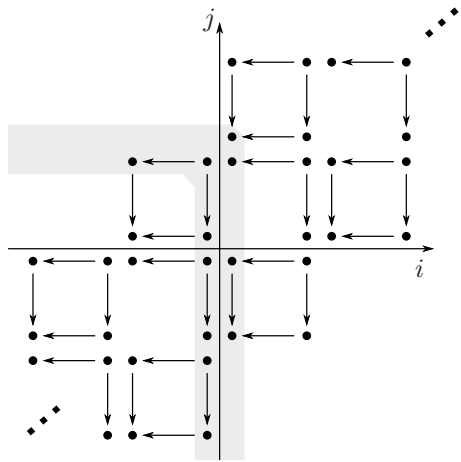}
\begin{picture}(0,0)
\put(-100,165) {$j$}
\put(-20,70) {$i$}
\end{picture}
\caption{Example: $CFK^{\infty}(S^{3},K)$ and $\widehat{A}_1$ for $K=\overline{6}_2$} 
\label{fig:complex-example}
\end{center}
\end{figure} 

\end{example}

\section{Chirally cosmetic surgery constraint}

In this section, by using the rank of the Heegaard Floer homology and constraints from the degree two finite type invariants we prove various chirally cosmetic surgery constraints in the introduction. 
In the following, by taking the mirror image if necessary, we assume that $\nu(K)\geq \nu(\overline{K})$ and $\nu(K)\geq 0$, as in Section \ref{section:HF}.

To begin with, let us observe that Proposition \ref{proposition:HFrank} immediately provides the following.

\begin{proposition}
\label{prop:rankHF-result}
Let $K$ be a non-trivial knot. If $S^{3}_K(m\slash n) \cong -S^{3}_K(m\slash n')$ ($\frac{m}{n} > \frac{m}{n'}$), then the following holds.
\begin{enumerate}
\item[(i)] $\frac{m}{n} > 0$.
\item[(ii)] If $\frac{m}{n'} >0$, then $S^{3}_K(m\slash n)$ is an L-space and $C_K=0$. In particular, $\frac{m}{n}, \frac{m}{n'}\geq 2g(K)-1$.
\item[(iii)] If $0 >\frac{m}{n'}$ and $\nu(K)=0$, then $n+n' =0$.
\item[(iv)] If $0 >\frac{m}{n'}$ and $\nu(K)>0$, then $n+n'>0$. Moreover,
\begin{itemize}
\item[(iv-a)] $\displaystyle \frac{n+n'}{m}=\frac{2}{4\nu(K)-2+C_K}$ when $\frac{m}{n} \leq 2\nu(K)-1$.
\item[(iv-b)] $\displaystyle \frac{n+n'}{m}=\frac{(4\nu(K)-2)(-n')}{mC_K}$ when $\frac{m}{n} \geq 2\nu(K)-1$ (in this case $C_K \neq 0$).
\end{itemize}
\end{enumerate}
\end{proposition}

In particular, this proves the nu-invariant constraint.

\begin{proof}[Proof of Theorem \ref{theorem:nu}]
If $K$ admits a chirally cosmetic surgery of $+$-type, by Proposition \ref{prop:rankHF-result} (ii) $C_K=0$. Thus by (\ref{eqn:genus}) $\nu(K)=g(K)>0$. If $K$ admits a chirally cosmetic surgery of $-$-type, by Proposition \ref{prop:rankHF-result} (iii) $\nu(K) > 0$.
\end{proof}

We reformulate Proposition \ref{prop:rankHF-result} in the following more convenient and informative form.

\begin{corollary}
\label{cor:HFbound}
Let $K$ be a knot in $S^{3}$. Assume that $K$ admits a chirally cosmetic surgery $S^{3}_K(\frac{m}{n}) \cong -S^{3}_K(\frac{m}{n'})$. Then
\[\left|\frac{n+n'}{m}\right| <
\begin{cases}
\frac{2}{2g(K)-1}& (C_K=0) \\ 
\frac{2}{C_K} & (C_K>0) 
\end{cases} \]
\end{corollary}
\begin{proof} 

If the surgery is of $0$-type, the assertion is trivial.
If the surgery is of $+$-type, by Corollary \ref{prop:rankHF-result} (ii) $\frac{m}{n},\frac{m}{n'} \geq 2g(K)-1$. Since $n\neq n'$, $\frac{n+n'}{m}<\frac{2}{2g(K)-1}$.

Finally we assume that the surgery is of $-$-type, so $\frac{m}{n} > 0 >\frac{m}{n'}$.
 By Proposition \ref{prop:rankHF-result}, if $\nu(K)=0$, then $\frac{n+n'}{m}=0$ so we have nothing to prove. Thus we assume that $\nu(K)>0$.

If $\frac{m}{n} \leq 2\nu(K)-1$, then $\frac{n+n'}{m} = \frac{2}{4\nu(K)-2+C_K}$.
When $C_K=0$, $g(K)=\nu(K)$ hence 
\[ \frac{n+n'}{m} =\frac{1}{2\nu(K)-1}=\frac{1}{2g(K)-1} <\frac{2}{2g(K)-1}. \]
When $C_K\neq 0$, then
\[ 0 < \frac{n+n'}{m} = \frac{2}{4\nu(K)-2+C_K} < \frac{2}{C_K}.\]

If $\frac{m}{n} \geq 2\nu(K)-1$, since $n+n'>0$, $\frac{-n'}{n}<1$. Hence
\[ \frac{n+n'}{m}=\frac{(4\nu(K)-2)(-n')}{mC_K} \leq \frac{2(-n')}{C_K n}< \frac{2}{C_K}.\]

\end{proof}

Corollary \ref{cor:HFbound} is useful; first of all, this excludes the possibility for chirally cosmetic surgery of $\pm$-type yielding integral homology spheres and homology projective space as we have mentioned in introduction.

\begin{proof}[Proof of Corollary \ref{cor:no-ZHS}]
Note that Corollary \ref{cor:HFbound} implies $|\frac{n+n'}{m}| < 1$ unless $g(K)=1$ and $C_K=0$. However $g(K)=1$ and $C_K=0$ implies that $K$ is the trefoil. Thanks to the classification of chirally cosmetic surgeries on the trefoil, $|\frac{n+n'}{m}| < 1$ also holds in this case. 
When $m=2$, $n,n'$ must be odd so $|n+n'|=1$ never happens.
\end{proof}

Moreover as we have seen and discussed in \cite{it1,iis}, various other invariants of 3-manifolds provide formulae of the form
\begin{align*}
&S^{3}_K\left(\frac{m}{n}\right) \cong -S^{3}_K\left(\frac{m}{n'}\right) \\
& \Rightarrow \frac{n+n'}{m} = \mbox{(Some specific value(s) determined by } K  (\mbox{and } m,n,n'))
\end{align*}
(see for example, Theorem \ref{theorem:deg-two-criterion} (ii)).

\begin{proof}[Proof of Theorem \ref{theorem:weak}]
First we assume that $g(K)\neq \nu(K)$. By Lemma \ref{lemma:C_K}, in this case $C_K\geq 4$, or, $g(K)=1$. 
Since $\nu(K) \leq g(K)$, when $g(K)=1$, $g(K)\neq \nu(K)$ implies that $\nu(K)=0$. Thus by Theorem \ref{theorem:nu}, in this case $K$ does not admit chirally cosmetic surgery of $\pm$-type.
If $C_K\geq 4$, by Corollary \ref{cor:HFbound} and Theorem \ref{theorem:deg-two-criterion}
\[ \left|\frac{n+n'}{m} \right|= \left|\frac{8v_3(K)}{7a_2(K)^{2}-a_2(K)-10a_4(K)} \right| < \frac{2}{C_K} \leq \frac{1}{2}\]
as desired.
When $g(K)= \nu(K)$, the assertion follows from Theorem \ref{theorem:deg-two-criterion} and Corollary \ref{cor:HFbound} as well.
\end{proof}

\begin{proof}[Proof of Theorem \ref{theorem:criteria-thin}]
By Lemma \ref{lemma:thin-C_K}, Corollary \ref{cor:HFbound}, and Theorem \ref{theorem:deg-two-criterion}, when $S^{3}_K(\frac{m}{n}) \cong -S^{3}_K(\frac{m}{n'})$,
\[ \left|\frac{n+n'}{m} \right|= \left|\frac{8v_3(K)}{7a_2(K)^{2}-a_2(K)-10a_4(K)} \right| < \frac{2}{C_K}=\frac{4}{|\det(K)|-2|\tau(K)|-1}.\]
\end{proof}

\setcounter{section}{1}
\renewcommand{\thesection}{\Alph{section}}
\section*{Appendix: Formula of $v_5$}

In this appendix we give a formula of the finite type invariant $v_5$ in terms of the Kauffman polynomial.

Throughout the appendix, we assume some familiarity of the theory of the LMO and the Kontsevich invariant. See \cite{oh}, for the basics of the LMO and the Kontsevich invariant.

In \cite{it1}, we gave a formula of the degree two part $\lambda_2$ and the degree three part $\lambda_3$ of the LMO invariant, the coefficients of the Jacobi diagram $\Thetatwodia$ and $\Thetathreedia$ for $S^{3}_K(r)$. We gave a formula $\lambda_2$ and $\lambda_3$ in terms or the surgery slope $r$ and the finite type invariants $v_2,v_3,v_4,w_4,v_5,v_6$ of $K$, which are the coefficients of the Jacobi diagrams 
\[ \Atwodia,\Btwodia, \Afourdia, \Ctwodia,\Bfourdia,\Asixdia \]
 of the (wheeled) Kontsevich invariant of $K$.

We also gave a formula of the finite type invariants $v_2,v_3,v_4,w_4,v_6$ in terms of the Conway and the Jones polynomial. However, we did not give an explicit formula of $v_5$, because $v_5$ is not determined by the Conway polynomial and the Jones polynomial, and, in the applications discussed in \cite{it1} we do not need computations of the actual value of $v_5$.

In \cite{it1} we used the $\mathfrak{sl}_2$ weight system evaluation to get a formula of $v_3,w_4$. A similar argument using $\mathfrak{so}_n$ weight system works for $v_5$, but here we give a slightly different approach.

A finite type invariant $v$ of degree $n$ is \emph{canonical} if there is a weight system $w:\mathcal{B} \rightarrow \C$ of degree $n$ (i.e. $w(D)=0$ whenever the degree of $D$ is not equal to $n$) such that $v(K)=w(Z^{\sigma}(K))$. Here $\mathcal{B}$ denotes the space of the open Jacobi diagrams and $Z^{\sigma}(K) \in \mathcal{B}$ denotes the Kontsevich invariant (taken so that it takes value in $\mathcal{B}$ by composing the inverse of the PBW isomorphism $\sigma$). In other words, $v$ is canonical if it is written as a linear combination of the coefficients of the degree $n$ part of the Kontsevich invariant. Thus by definition, $v_5$ is a canonical finite type invariant of degree $5$.

On the other hand, we can extract canonical finite type invariants from the Kauffman polynomial as follows.

\begin{lemma}
Let $F_K(a,z)$ be the Kauffman polynomial of a knot $K$ and $N \geq 1$ be a positive integer. We expand $F_K(a,z)$ as a formal power series of $h$

\begin{equation}
\label{eqn:expansion}
F_K(ie^{Nh}, -i(e^{h}-e^{-h})) = \sum_{h\geq 0}k_{n,N}(K)h^{n}
\end{equation}
where $i=\sqrt{-1}$.
Then $k_{n,N}(K)$ is a canonical finite type invariant of degree $n$.
\end{lemma}
\begin{proof}
Let $D_K(a,z)$ be the Dubrovnik polynomial, a version of the Kauffman polynomial given by 
\[ D_K(a,z)=(-1)^{\# K-1}F_K(ia,-iz). \]
Here $\#K$ denotes the number of component of $K$.
It is known that $D_K(e^{Nh}, e^{h}-e^{-h})$ is equal to the quantum $(\mathfrak{so}_{4N+1},V)$-invariant $Q^{\mathfrak{so}_{4N+1};V}(K)$ \cite{mps}, where $V$ is the standard representation of $\mathfrak{so}_{4N+1}$. Thus for each $N$, the expansion (\ref{eqn:expansion}) is nothing but an expansion of the quantum $(\mathfrak{so}_{4N+1},V)$-invariant 
\begin{align*}
Q^{\mathfrak{so}_{4N+1};V}(K) &= \sum_{h\geq 0}k_{n,N}(K)h^{n}. 
\end{align*}
By \cite[Theorem XX.8.3]{ka}, $k_{n,N}$ is a canonical finite type invariant of degree $n$.
\end{proof}

It is known that the dimension of the degree $5$ part of $\mathcal{B}$ is four.
By computing $k_{5,2},k_{5,3},k_{5,4},k_{5,5}$ for several knots, one can confirm they are linearly independent so they form a basis of canonical finite type invariants of degree $5$.
Therefore we may write 
\[ v_5=x_1k_{5,2} + x_2k_{5,3} + x_3 k_{5,4}+ x_4 k_{5,5} \]
where $x_1,x_2,x_3,x_4 \in \R$ are constants.

The coefficients $x_1,x_2,x_3,x_4$ are determined by computing
$v_5,k_{5,2},k_{5,3},k_{5,4},k_{5,5}$ for suitable knots.
Since we can compute the Kauffman polynomial, the computation of $k_{5,2},k_{5,3},k_{5,4},k_{5,5}$ are routine.

To compute $v_5$ for some specific knots, we use the reduced 2-loop polynomial $\widehat{\Theta}_K(t)$ of $K$.
The reduced 2-loop polynomial is an invariant of a knot that comes from the $\mathfrak{sl}_2$ weight system evaluation of the two-loop part of the Kontsevich invariant. For details of the 2-loop polynomial we refer to \cite{oh2,oh3}.

We express $v_5$ in terms of the reduced 2-loop polynomial $\widehat{\Theta}_K(t)$.

\begin{lemma}
\label{lemma:v_5-2-loop}
Let $L_K(x)=\frac{\widehat{\Theta}_K(e^x)}{\Delta_K(e^x)^2}$, where $\Delta_K(t)$ is the Alexander polynomial, normalized so that $\Delta_K(t)=\Delta_K(t^{-1})$, $\Delta_K(1)=1$.
Then
\[v_5(K) = \frac{1}{16}\left(\frac{1}{6}L_K(0) + L_K''(0) \right)\]
In particular, when $\Delta_K(t)=1$ then
\[v_5(K) =\frac{1}{16}\left(\frac{1}{6}\widehat{\Theta}_K(1) + \widehat{\Theta}_K''(1) \right)\]
\end{lemma} 
\begin{proof}
This follows from the same argument as \cite[Proposition 5.1]{oh2}, where it was proven that $v_3(K)=\frac{1}{8}\widehat{\Theta}(1)$.

Let $W$ be the map that sends a labelled 2-loop Jacobi diagram to $\Q[[x]]$, given by
\[ W\left(\raisebox{-8mm}{
\begin{picture}(60,40)
\put(30,25){\oval(60,35)}
\put(0,25){\line(1,0){60}}
\put(20,48){\footnotesize $f_1(x)$}
\put(20,29){\footnotesize $f_2(x)$}
\put(20,10){\footnotesize $f_3(x)$}
\end{picture}
}\right) = \frac{1}{6}\sum_{\{i,j,k\}=\{1,2,3\}}f_i(x)f_j(-x)f_k(0)\]
Here $f_i(x) \in \Q[[x]]$ is a formal power series of $x$, and a Jacobi diagram whose edge is labeled by a power series $f(x)= c_0+c_1x+ c_2x^{2}+c_3x^{3}+\cdots $ represents the Jacobi diagram 
\[ \raisebox{-4mm}{
\begin{picture}(20,30)
\put(2,0){\line(0,1){30}}
\put(5,22){\scriptsize $f(x)$}
\end{picture} 
}
=
c_0 \raisebox{-4mm}{
\begin{picture}(10,30)
\put(2,0){\line(0,1){30}}
\end{picture} } 
+ c_1 \raisebox{-4mm}{
\begin{picture}(10,30)
\put(2,0){\line(0,1){30}}
\put(2,15){\line(1,0){10}}
\end{picture} 
}
+ c_2\raisebox{-4mm}{
\begin{picture}(10,30)
\put(2,0){\line(0,1){30}}
\put(2,12){\line(1,0){10}}
\put(2,18){\line(1,0){10}}
\end{picture} 
}
 +c_3 \raisebox{-4mm}{
\begin{picture}(10,30)
\put(2,0){\line(0,1){30}}
\put(2,9){\line(1,0){10}}
\put(2,15){\line(1,0){10}}
\put(2,21){\line(1,0){10}}
\end{picture} 
} + \cdots.
\]

The map $W$ sends the 2-loop part of the logarithm of the Kontsevich invariant to 
\[ \frac{1}{12}(e^{\frac{x}{2}}-e^{-\frac{x}{2}})^2L_K(x) \] 
which is viewed as a formal power series by taking the Taylor expansion at $x=0$;
\[ \frac{1}{12}(e^{\frac{x}{2}}-e^{-\frac{x}{2}})^2L_K(x) = c_0 + c_1x + c_2 x^{2}+c_3x^{3}+x_4x^{4}+\cdots. \]

Since $v_5$ is the defined as the coefficient of the 2-loop Jacobi diagram $\Bfourdia$ and 
\[ W\left(\Bfourdia \right)=W\left(
\raisebox{-8mm}{
\begin{picture}(50,40)
\put(25,25){\oval(50,35)}
\put(0,25){\line(1,0){50}}
\put(20,46){\footnotesize $x^4$}
\put(20,29){\footnotesize $1$}
\put(20,10){\footnotesize $1$}
\end{picture}
}\right) = \frac{2}{3}x^{4},
\]
we conclude that 
\begin{align*}
v_5 &= \frac{3}{2}c_4 = \frac{3}{2}\left(\frac{1}{4!}\frac{d^4}{dx^4}\left.\left(\frac{1}{12}(e^{\frac{x}{2}}-e^{-\frac{x}{2}})^2L_K(x) \right)\right|_{x=0}\right) \\
&=\frac{1}{16}\left(\frac{1}{6}L_K(0) + L_K''(0) \right)
\end{align*}
\end{proof}

The reduced 2-loop polynomial is computed for the torus knot $T(p,q)$ \cite[Corollary 3.1]{oh2} and the genus one knots \cite[Corollary 3.1]{oh2}.
Using these results and Lemma \ref{lemma:v_5-2-loop}, we compute values of $v_5, k_{5,2}, k_{5,3}, k_{5,4}, k_{5,5}$ for several knots. The results are presented in Table \ref{table:degree-five-invariant}. Here $ J(\ell,m)$ denotes the double twist knot.

\begin{table}[h]
  \centering
  \begin{tabular}{|c|c|c|c|c|c|}
   \hline
   Knot & $v_5$ & $k_{5,2}$ & $k_{5,3}$ & $k_{5,4}$ & $k_{5,5}$ \\
   \hline \hline
   $T(2,3)$ & $-17/48$ &  $-176$ & $-736$  & $-1056$ & $1280$ \\
   $T(2,5)$ & $-229/48$ & $-2480$ & $-11360$ & $-21024$ & $-3840$ \\
   $J(-2,4)$ & $37/16$ & $1104$  & $4128$ & $3552$ & $-19200$ \\
   $J(2,4)$ & $43/48$ & $464$ & $2720$ & $9696$ & $26880$ \\
      \hline
  \end{tabular}
\medskip
\caption{Values of the finite type invariants $v_5, v_{5,2}, v_{5,3}, v_{5,4}$ and $v_{5,5}$ for some specific knots}
 \label{table:degree-five-invariant}
\end{table}

From this computation, we get an explicit formula of $v_5$ in terms of the Kauffman polynomial.
\begin{theorem}
\label{theorem:v5-formula}
\[ v_5(K)=\dfrac{1}{768} k_{5,2}(K) + \dfrac{1}{768} k_{5,3}(K) - \dfrac{1}{1536} k_{5,4}(K)+ \dfrac{7}{61440} k_{5,5}(K) \]
\end{theorem}

\end{document}